\documentclass[preprint,12pt,authoryear]{elsarticle}
\setcitestyle{numbers,square,comma}

\usepackage{amssymb}
 \usepackage{amsthm}
 \usepackage{amsmath}
 \usepackage{hyperref}
 \hypersetup{hidelinks}
 \usepackage{verbatim}
 \usepackage{tikz}
\usepackage{tikz-cd}

\journal{Designs, Codes and Cryptography}

\begin{document}

\begin{frontmatter}



\title{Resolving a conjecture on permutation polynomials over $\mathbb{F}_{2^n}$}


\author[4]{Yi Li}



\affiliation[4]{organization={School of Computer Science, Shanghai Jiao Tong University}, city={Shanghai}, postcode={200240}, country={China}}
\begin{abstract}

Let $\delta\in\mathbb{F}_{2^n}$ satisfy
$\operatorname{Tr}_{\mathbb{F}_{2^n}/\mathbb{F}_2}(\delta)=1$. We study the permutation behavior of
$$
f(x)
=
\left(\frac{1}{x^2+x+\delta}\right)^{2^k}+x
$$
over $\mathbb{F}_{2^n}$. Helleseth and Zinoviev proved that
$f(x)$ is a permutation for $k=0,1$, and remarked that numerical evidence suggests that no other cases occur. In this paper, we confirm their assertion by proving that, for $0\leq k<n$, $f(x)$ is a permutation of $\mathbb{F}_{2^n}$ if and only if $k=0$ or $k=1$.

\end{abstract}


\end{frontmatter}


\newtheorem{exa}{Example}
\newtheorem{rmk}{Remark}
\newtheorem{defn}{Definition}
\newtheorem{pro}{Proposition}
\newtheorem{thm}{Theorem}
\newtheorem{lem}{Lemma}
\newtheorem{case}{Case}	
\newtheorem{cor}{Corollary}
\newtheorem{subcase}{Case}
\numberwithin{subcase}{case}
\newcommand{\F}{\mathbb F}
\newcommand{\Tr}{\operatorname{Tr}}
\section{Introduction}

Let $q$ be a prime power and let $\mathbb{F}_q$ denote the finite field with
$q$ elements. A polynomial $f\in \mathbb{F}_q[x]$ is called a permutation
polynomial of $\mathbb{F}_q$ if the map $x\mapsto f(x)$ is a bijection on
$\mathbb{F}_q$. Permutation polynomials are classical objects in finite field
theory and have many applications in coding theory \cite{OTCCFM2013,PPACC2007}, cryptography \cite{PKE1998,RSA1978},
combinatorial design theory \cite{AFSHDS2006}, and we refer the reader to \cite{PPOFFASORA2015,HFF2013,PPOFFAIA2019, wangDCC2024} and references therein for more details of the recent advances and contributions of the area.

Permutation polynomials involving terms of the form $(x^2+x+\delta)^{-1}$ first arose in the study of Kloosterman sums over binary finite fields. Helleseth and Zinoviev \cite{HZ2003} investigated permutation polynomials of the form
$$
\left(\frac{1}{x^2+x+\delta}\right)^s+x
$$
over $\mathbb{F}_{2^n}$, where $\operatorname{Tr}_{\mathbb{F}_{2^n}/\mathbb{F}_2}(\delta)=1$. This trace condition guarantees that $x^2+x+\delta$ has no roots in $\mathbb{F}_{2^n}$, and hence the reciprocal expression is well defined on the whole field. Their work was motivated by the derivation of new identities for Kloosterman sums over binary finite fields, and the cases $s=1$ and $s=2$ were established there.

At the end of their paper, Helleseth and Zinoviev remarked that numerical evidence suggests that the polynomial
$$
f(x)
=
\left(\frac{1}{x^2+x+\delta}\right)^{2^k}+x
$$
over $\mathbb{F}_{2^n}$ is a permutation polynomial only in the cases
$k=0$ and $k=1$. This remark naturally raises the problem of determining whether these are indeed the only possible cases.

In this paper, we give a complete answer to this problem. More precisely, under the assumption
$$
\operatorname{Tr}_{\mathbb{F}_{2^n}/\mathbb{F}_2}(\delta)=1,
$$
we prove that
$$
f(x)
=
\left(\frac{1}{x^2+x+\delta}\right)^{2^k}+x
$$
is a permutation polynomial of $\mathbb{F}_{2^n}$ if and only if
$k=0$ or $k=1$, where $0\leq k<n$.

The rest of this paper consists of the proof of the main theorem. In Section 2, we derive the collision equation for
$$
f(x)
=
\left(\frac{1}{x^2+x+\delta}\right)^{2^k}+x
$$
and prove the necessary and sufficient condition for this function to be a permutation of $\mathbb{F}_{2^n}$. Section 3 concludes this paper.

\section{Proof of the main theorem}

In this section we prove the main theorem.

\begin{thm}
    $f(x)=\left(\frac{1}{x^2+x+\delta}\right)^{2^k}+x$ is a permutation polynomial over $\mathbb{F}_{2^n}$ if and only if $k\in \{0,1\}$ when $\text{Tr}_{\mathbb{F}_{2^n}/\mathbb{F}_2}(\delta)=1$.
\end{thm}

\begin{proof}
    Let $z=x^2+x+\delta$. Assume $f(x)$ is not a permutation polynomial over $\mathbb{F}_{2^n}$, then there exists $u\neq 0$, $f(x)=f(x+u)$. Expanding this, we get
    \begin{eqnarray}\label{1}
        x+\frac{1}{z^{2^k}}=x+u+\frac{1}{(z+u^2+u)^{2^k}}.
    \end{eqnarray}

    Eliminating $x$, we get
    \begin{eqnarray}\label{2}
        u=\frac{1}{z^{2^k}}+\frac{1}{(z+u^2+u)^{2^k}}.
    \end{eqnarray}

    Raising both sides of the equation to the $2^{-k}$-th power, we obtain 
    \begin{eqnarray}\label{3}
        u^{2^{-k}}=\frac{1}{z}+\frac{1}{z+u^2+u}.
    \end{eqnarray}

    Rewriting Equation (\ref{3}), we obtain a quadratic equation in the variable $z$, 
    \begin{eqnarray}\label{4}
        z^2+(u^2+u)z+\frac{u^2+u}{u^{2^{-k}}}=0.
    \end{eqnarray}

    The function $f(x)$ admits a collision if and only if Equation (\ref{4}) has a solution in $\mathbb{F}_{2^n}$, and any solution $z$ necessarily satisfies $\text{Tr}_{\mathbb{F}_{2^n}/\mathbb{F}_2}(z)=1$.

    Since $A=u^2+u\neq 0$, let $z=Aw$. Substituting into Equation (\ref{4}) yields $(u^2+u)^2w^2+(u^2+u)^2w+\frac{u^2+u}{u^{2^{-k}}}=0$. Dividing both sides by $A^2$, we obtain
    \begin{eqnarray}\label{5}
        w^2+w=\frac{1}{u^{2^{-k}}(u^2+u)}.
    \end{eqnarray}

    Since $z=(u^2+u)w$, we know $\text{Tr}_{\mathbb{F}_{2^n}/\mathbb{F}_2}(z)=\text{Tr}_{\mathbb{F}_{2^n}/\mathbb{F}_2}(u^2w)+\text{Tr}_{\mathbb{F}_{2^n}/\mathbb{F}_2}(uw)$. Since $\text{Tr}_{\mathbb{F}_{2^n}/\mathbb{F}_2}(uw)=\text{Tr}_{\mathbb{F}_{2^n}/\mathbb{F}_2}(u^2w^2)$, we get $\text{Tr}_{\mathbb{F}_{2^n}/\mathbb{F}_2}(z)=\text{Tr}_{\mathbb{F}_{2^n}/\mathbb{F}_2}(u^2(w^2+w))$. From Equation (\ref{5}), we get
    \begin{eqnarray}\label{6}
        \text{Tr}_{\mathbb{F}_{2^n}/\mathbb{F}_2}(z)=\text{Tr}_{\mathbb{F}_{2^n}/\mathbb{F}_2}\left(\frac{u^2}{u^{2^{-k}}(u^2+u)}\right)=\text{Tr}_{\mathbb{F}_{2^n}/\mathbb{F}_2}\left(\frac{u}{u^{2^{-k}}(u+1)}\right).
    \end{eqnarray}

    Now we show when $k=0$ or $1$, we can get a contradiction.

    When $k=0$, Equation (\ref{4}) has a solution if and only if $\text{Tr}_{\mathbb{F}_{2^n}/\mathbb{F}_2}\left(\frac{1}{u^2(u+1)}\right)=0$. Using a simple partial fraction decomposition $\frac{1}{u^2(u+1)}=\frac{1}{u^2}+\frac{1}{u}+\frac{1}{u+1}$, we obtain 
    \begin{eqnarray}\label{7}
        \text{Tr}_{\mathbb{F}_{2^n}/\mathbb{F}_2}\left(\frac{1}{u^2(u+1)}\right)=\text{Tr}_{\mathbb{F}_{2^n}/\mathbb{F}_2}\left(\frac{1}{u^2}\right)+\text{Tr}_{\mathbb{F}_{2^n}/\mathbb{F}_2}\left(\frac{1}{u}\right)+\text{Tr}_{\mathbb{F}_{2^n}/\mathbb{F}_2}\left(\frac{1}{u+1}\right).
    \end{eqnarray}

    Hence Equation (\ref{4}) has a solution if and only if $\text{Tr}_{\mathbb{F}_{2^n}/\mathbb{F}_2}\left(\frac{1}{u+1}\right)=0$.

    However, in this case $\text{Tr}_{\mathbb{F}_{2^n}/\mathbb{F}_2}(z)=\text{Tr}_{\mathbb{F}_{2^n}/\mathbb{F}_2}\left(\frac{1}{u+1}\right)=0$, which is a contradiction since $\text{Tr}_{\mathbb{F}_{2^n}/\mathbb{F}_2}(z)=1$.

    When $k=1$, $\text{Tr}_{\mathbb{F}_{2^n}/\mathbb{F}_2}\left(z\right)=\text{Tr}_{\mathbb{F}_{2^n}/\mathbb{F}_2}\left(\frac{u}{u^{\frac{1}{2}}(u+1)}\right)=\text{Tr}_{\mathbb{F}_{2^n}/\mathbb{F}_2}\left(\frac{u^2}{u(u+1)^2}\right)=\text{Tr}_{\mathbb{F}_{2^n}/\mathbb{F}_2}\left(\frac{u}{u^2+1}\right)$. Again, using a simple partial fraction decomposition $\frac{u}{u^2+1}=\frac{1}{u+1}+\frac{1}{u^2+1}$, we get 
    \begin{eqnarray}\label{8}
        \text{Tr}_{\mathbb{F}_{2^n}/\mathbb{F}_2}\left(z\right)=\text{Tr}_{\mathbb{F}_{2^n}/\mathbb{F}_2}\left(\frac{1}{u+1}\right)+\text{Tr}_{\mathbb{F}_{2^n}/\mathbb{F}_2}\left(\frac{1}{u^2+1}\right)=0.
    \end{eqnarray}

    This contradicts to the condition $\text{Tr}_{\mathbb{F}_{2^n}/\mathbb{F}_2}(z)=1$.

    Next, we show when $k\geq 2$, we can find such $u$, which satisfy
    \begin{eqnarray}\label{9}
        T(u)=\text{Tr}_{\mathbb{F}_{2^n}/\mathbb{F}_2}\left(\frac{1}{u^{2^{-k}}(u^2+u)}\right)=0,
    \end{eqnarray}
    \begin{eqnarray}\label{10}
        S(u)=\text{Tr}_{\mathbb{F}_{2^n}/\mathbb{F}_2}\left(\frac{u}{u^{2^{-k}}(u+1)}\right)=1.
    \end{eqnarray}

        Let $x=\frac{u}{u+1}$. Since $u\not\in \{0,1\}$, $\frac{u}{u+1}$ is a bijection over $\mathbb{F}_{2^n}\setminus \{0,1\}$. Plugging $x$ into $T(u)$ and $S(u)$, we get

        \begin{eqnarray}\label{11}
            T(x)=\text{Tr}_{\mathbb{F}_{2^n}/\mathbb{F}_2}\left(\frac{1}{(x/(x+1))^{2^{-k}}(x/(x+1))(1/(x+1))}\right)=\text{Tr}_{\mathbb{F}_{2^n}/\mathbb{F}_2}\left((x+1)^{2^{-k}+2}x^{-1-2^{-k}}\right),
        \end{eqnarray}

        \begin{eqnarray}\label{12}
            S(x)=\text{Tr}_{\mathbb{F}_{2^n}/\mathbb{F}_2}\left(\frac{x/(x+1)}{(x/(x+1))^{2^{-k}}(1/(x+1))}\right)=\text{Tr}_{\mathbb{F}_{2^n}/\mathbb{F}_2}\left(x^{1-2^{-k}}(x+1)^{2^{-k}}\right).
        \end{eqnarray}

         Raising both sides of Equation (\ref{12}) to the $2^k$-th power, we get

         \begin{eqnarray}\label{13}
             S(x)=\text{Tr}_{\mathbb{F}_{2^n}/\mathbb{F}_2}(x^{2^k-1}(x+1))=\text{Tr}_{\mathbb{F}_{2^n}/\mathbb{F}_2}(x^{2^k}+x^{2^k-1})=\text{Tr}_{\mathbb{F}_{2^n}/\mathbb{F}_2}\left(x+x^{2^{k}-1}\right).
         \end{eqnarray}

         Again, raising both sides of Equation (\ref{11}) to the $2^k$-th power, we get
         \begin{eqnarray}\label{14}
             T(x)=\text{Tr}_{\mathbb{F}_{2^n}/\mathbb{F}_2}\left((x+1)^{2^{k+1}+1}x^{-2^k-1}\right)=\text{Tr}_{\mathbb{F}_{2^n}/\mathbb{F}_2}\left(x^{2^k}+x^{2^{k}-1}+x^{-2^{k}}+x^{-2^k-1}\right).
         \end{eqnarray}

         Hence the condition $T(x)=0$ and $S(x)=1$ are equivalent to $S(x)=1$ and $G(x)=1$, where $G(x)=\text{Tr}_{\mathbb{F}_{2^n}/\mathbb{F}_2}(x^{-1}+x^{-2^k-1})$.

         Now we want to find such $x$, which satisfy
         \begin{eqnarray}\label{15}
             S(x)=\text{Tr}_{\mathbb{F}_{2^n}/\mathbb{F}_2}(x+x^{2^k-1})=1,
         \end{eqnarray}

         and \begin{eqnarray}\label{16}
              G(x)=\text{Tr}_{\mathbb{F}_{2^n}/\mathbb{F}_2}(x^{-1}+x^{-2^k-1})=1.
         \end{eqnarray}

Let $F(x)=S(x)G(x)$. If $f(x)$ is a permutation over $\mathbb{F}_{2^n}$, then $F(x)=0$ for any $\mathbb{F}_{2^n}\setminus \{0,1\}$. Since $F(1)= 0$, $F(x)=0$ for all $x\in \mathbb{F}_{2^n}^{*}$. This is equivalent to saying that $F(x)$ is the zero polynomial when viewed in $\mathbb{F}_2[x]/(x^{2^{n}-1}-1)$. Next, we show that when $k\geq 2$, the polynomial $F(x)$ is not the zero polynomial for any $n$, thereby completing the proof.

Case $1$, when $n$ is odd, we now examine the coefficient of $x^0$, i.e., the constant term. The constant term equals the number of pairs $(i,j)$ satisfying $A_i\equiv B_j \pmod {2^{n}-1}$. There are four possible combinations:
$2^i\equiv 2^j \pmod {2^n-1}$, $(2^k-1)2^i\equiv 2^j \pmod {2^n-1}$, $2^i\equiv (2^k+1)2^j \pmod {2^n-1}$, $(2^k-1)2^i\equiv (2^k+1)2^j\pmod {2^n-1}$. Let us check the cases one by one. 

1. $2^i\equiv 2^j \pmod{2^n-1}$. Clearly, this congruence has exactly $n$ solutions, namely when $i=j$.

2. $(2^k-1)2^i\equiv 2^j \pmod{2^n-1}$. This congruence equation is equivalent to $(2^k-1)\equiv 2^{j-i}\pmod {2^n-1}$. Since $k\geq 2$ and $k<n$, $2^k-1$ is an odd integer greater than or equal to $3$. It cannot be a power of $2$; therefore, there are no solutions.

3. $2^i\equiv (2^k+1)2^j\pmod {2^n-1}$. This congruence equation is equivalent to $2^{i-j}\equiv 2^k+1 \pmod{2^n-1}$. Since $k\geq 2$ and $k<n$, $2^k+1$ is an odd integer. It cannot be a power of $2$; therefore, there are no solutions.

4. $(2^k-1)2^i\equiv (2^k+1)2^j\pmod {2^n-1}$. This congruence equation is equivalent to $2^k-1\equiv (2^k+1)2^m$, where $m=j-i$. We rewrite the equation as $2^{k+m}+2^m+1\equiv 2^k \pmod{2^n-1}$. Since the left-hand side of the congruence is a sum of three powers of $2$, in order for it to reduce to a single power of $2$ modulo $2^n-1$
(namely the $2^k$ on the right-hand side), a carry must occur on the left-hand side (that is, at least two of the exponents must coincide).

If $m\equiv 0 \pmod n$, then the left-hand side becomes $2^k+1+1=2^k+2$, while the right-hand side is $2^k$, which is impossible.

If $k+m\equiv m\pmod n$, then $k\equiv 0 \pmod n$, which is a contradiction since $k\geq 2$.

If $k+m\equiv 0\pmod n$, then $2^{-k}+2\equiv 2^k \pmod{2^n-1}$. This congruence equation is equivalent to $2^{k+1}+1\equiv 2^{2k}\pmod{2^n-1}$. This congruence equation holds if and only if $2^{k+1}+1+2^n-1=2^{2k}$. This equation is equivalent to $2^{k+1}+2^n=2^{2k}$, which implies $k=n-1$ and $2k=n+1$. This happens only when $n=3$ and $k=2$. In other cases, the congruence equation has no solutions.

We sum the solutions of these four equations (excluding the case \(n=3, k=2\)). For any odd integer \(n \ge 3\), the total number of pairs contributing to the constant term \(x^0\) is
\[
n + 0 + 0 + 0 = n.
\]
Since \(n\) is odd, the coefficient of the constant term is therefore always \(1\). Hence, in this case, \(f(x)\) is not a permutation polynomial over \(\mathbb{F}_{2^n}\).

Case 2, when $n$ is even and $n\geq 4$, $k\geq 3$, we now examine the coefficient of $x^{2^k-2}$. This term equals the number of pairs $(i,j)$ satisfying $A_i-B_j\equiv 2^k-2\pmod{2^n-1}$. There are four possible combinations: $2^i-2^j\equiv 2^k-2 \pmod{2^n-1}$, $(2^k-1)2^i-2^j\equiv 2^k-2 \pmod{2^n-1}$, $2^i-(2^k+1)2^j\equiv 2^k-2\pmod{2^n-1}$, $(2^k-1)2^i-(2^k+1)2^j\equiv 2^k-2\pmod{2^n-1}$. Let us check the cases one by one.

1. $2^i-2^j\equiv 2^k-2\pmod{2^n-1}$. When $i\geq j$, $2^i-2^j\geq 0$. Since $i< n$, $2^i-2^j<2^n-1$. Hence, in this case, the congruence equation is in fact an equation. Since the binary representation of a number is unique, we conclude that the only solution is $i=k$, $j=1$. When $i<j$, $2^i-2^j$ is a negative number. The positive representative of the left-hand side congruence class is in fact $2^n-2^j+2^i-1$. Now this number and $2^k-2$ are all in $[0,2^n-2]$. Hence we have an equation $2^n-2^j+2^i-1=2^k-2$. We move the negative sign to the other side of the equation, obtaining a remarkably elegant equation consisting entirely of positive terms: $2^n+2^i+1=2^k+2^j$. Since $j,k\leq n-1$, this equation has no solutions. Hence, in this case, only one solution can exist.

2. $(2^k-1)2^i-2^j\equiv 2^k-2\pmod{2^n-1}$. We move the negative terms to the other side, obtaining a congruence equation consisting entirely of positive terms: $2^{k+i}+2\equiv 2^k+2^i+2^j\pmod{2^n-1}$. The binary weight of the left-hand side is at most 2 (since it is a sum of two powers of 2). The right-hand side contains three powers of 2; in order for it to equal the left-hand side, a carry must occur on the right-hand side. There are only three possible types of carry.

(1).$i=j$. In this case, the right-hand side is $2^k+2^{i+1}$. While the left-hand side is $2^{k+i}+2$. The only solution is $i=0$, $j=0$.

(2).$k=i$. In this case, the right-hand side is $2^{k+1}+2^j$, while the left-hand side is $2^{2k}+2$. The only possible solution is $j=1$, $2k\equiv k+1 \pmod{ n}$, which implies $k\equiv 1 \pmod{n}$, a contradiction due to $k\geq 3$.

(3).$k=j$. In this case, the right-hand side is $2^{k+1}+2^i$, while the left-hand side is $2^{k+i}+2$. The only possible solution is $i=1$.

In summary, in this case, there are two possible solutions.

3. $2^i-(2^k+1)2^j\equiv 2^k-2 \pmod{2^n-1}$. Similarly, we can obtain a congruence equation consisting entirely of positive terms $2^i+2\equiv 2^k+2^{k+j}+2^j\pmod{2^n-1}$. As in the discussion above, a carry must occur on the right-hand side. There are only two possible types of carry.

(1) $k\equiv k+j\pmod n$. In this case, the right-hand side is $2^{k+1}+1$. The only possible solution is $i=0$, $k\equiv 0\pmod n$, a contradiction, since $k\geq 3$.

(2) $k\equiv j\pmod n$. In this case, the right-hand side is $2^{k+1}+2^{2k}$. By comparing the exponents on both sides, we require $2k\equiv 1\pmod{n}$; however, this is impossible since $n$ is even.

In summary, in this case, there are no possible solutions.

4. $(2^k-1)2^i-(2^k+1)2^j\equiv 2^k-2 \pmod{2^n-1}$. Similarly, we can obtain a congruence equation consisting entirely of positive terms $2^{k+i}+2\equiv 2^i+2^j+2^{k+j}+2^k\pmod{2^n-1}$. As in the discussion above, two carries must occur on the right-hand side. Let us check the case one by one.

(1) $i\equiv j\pmod{n}$, in this case, the right-hand side is $2^{i+1}+2^{k+i}+2^k$. Then the congruence equation is $2^{i+1}+2^k\equiv 2\pmod{2^n-1}$, which implies $k\equiv 0\pmod{n}$, a contradiction.

(2) $i\equiv k+j\pmod{n}$, in this case, the right-hand side is $2^{k+j+1}+2^j+2^k$. If $k+j+1\equiv j\pmod{n}$, then the right-hand side is $2^{j+1}+2^k$. By comparing the exponents on both sides, we can get a contradiction. If $k+j+1\equiv k\pmod{n}$, then the right-hand side is $2^{k+1}+2^j$. Again, we can get a contradiction. If $j\equiv k\pmod{n}$, then the right-hand side is $2^{k+j+1}+2^{k+1}$. Again, we can get a contradiction.

(3) $i\equiv k\pmod{n}$, in this case, the right-hand side is $2^j+2^{k+j}+2^{k+1}$. If $j\equiv k+1\pmod{n}$, then the right-hand side is $2^{k+2}+2^{k+j}$. The only possible solution of $(i,j)$ is $(2,3)$ when $k=2$, $n=4$. If $j\equiv 1\pmod{n}$, then the right-hand side is $2^{k+2}+2$. In this case, $i=2$, $j=1$ is the only possible solution when $k\equiv 2\pmod{n}$.

(4) $j\equiv k+j\pmod{n}$, which is equivalent to $k\equiv 0\pmod{n}$, a contradiction.

(5) $j\equiv k\pmod{n}$. In this case, the right-hand side is $2^{i}+2^{k+j}+2^{k+1}$. If $i\equiv k+j\pmod{n}$, then the right-hand side is $2^{i+1}+2^{k+1}$. Again, we can get a contradiction. If $i\equiv k+1\pmod{n}$, then the right-hand side is $2^{k+2}+2^{k+j}$. Again, we can get a contradiction.

(6) $k+j\equiv k\pmod{n}$, in this case, the right-hand side is $2^i+1+2^{k+1}$. If $i\equiv 0\pmod{n}$, then the right-hand side is $2^{k+1}+2$. Again, we can get a contradiction. If $i\equiv k+1\pmod{n}$, then the right-hand side is $2^{k+2}+1$. Again, we can get a contradiction. If $k+1\equiv 0\pmod{n}$, then the right-hand side is $2^i+2$. Again, we can get a contradiction.

In summary, in this case, there are no possible solutions.

We sum the solutions of these four equations. For any even integer $n\geq 4$, $k\geq 3$, the total number of pairs contributing to the $x^{2^k-2}$ is 
\[
1+2+0+0=3.
\]

Since the characteristic of the quotient ring is $2$, the coefficient of $x^{2^{k}-2}$ is $1$. Hence, in this case, $f(x)$ is not a permutation polynomial over $\mathbb{F}_{2^n}$.

Case 3, when $n$ is even and $n\geq 6$, $k=2$, we now examine the coefficient of $x^{2^n-2}$. This term equals the number of pairs $(i,j)$ satisfying $A_i-B_j\equiv -1\pmod{2^n-1}$. There are four possible combinations: $2^i-2^j\equiv-1\pmod{2^n-1}$, $3\cdot2^i-2^j\equiv -1\pmod{2^n-1}$, $2^i-5\cdot 2^j\equiv -1\pmod {2^n-1}$, $3\cdot 2^i-5\cdot 2^j\equiv -1\pmod{2^n-1}$. Let us check the cases one by one.

1. $2^i-2^j\equiv-1 \pmod{2^n-1}$. We move the negative terms to the other side, obtaining a congruence equation consisting entirely of positive terms: $2^i+1\equiv 2^j\pmod{2^n-1}$. Since $i<n$ and $j<n$, the congruence equation is in fact an equation. The only solution is $i=0$, $j=1$.

2. $3\cdot 2^i-2^j\equiv -1\pmod{2^n-1}$. Similarly, we can obtain a congruence equation consisting entirely of positive terms: $2^i+2^{i+1}+1\equiv 2^j\pmod{2^n-1}$. If $i\equiv 0\pmod{n}$, then the only solution is $i=0$, $j=2$. If $i+1\equiv 0\pmod{n}$, then the left-hand side is $2^{n-1}+2$. Again, we can get a contradiction.

3. $2^i-5\cdot 2^j\equiv -1\pmod{2^n-1}$. Similarly, we can obtain a congruence equation consisting entirely of positive terms: $2^i+1\equiv 2^j+2^{j+2}\pmod{2^{n}-1}$. If $i\equiv j\pmod{n}$, then $j=n-2$. The only solution is $(n-2,n-2)$. If $j\equiv 0\pmod{n}$, then $i=2$. In this case, the only solution is $i=2$, $j=0$.

4. $3\cdot 2^i-5\cdot 2^j\equiv -1 \pmod{2^n-1}$. Similarly, we can obtain a congruence equation consisting entirely of positive terms: $2^i+2^{i+1}+1\equiv 2^j+2^{j+2}\pmod{2^n-1}$. If $i\equiv 0\pmod{n}$, then the left-hand side is $4$. Again, we can get a contradiction. If $i+1\equiv 0\pmod{n}$, then the left-hand side is $2^{n-1}+2$. Since $n\geq 6$, the only solution is $i=n-1$, $j=n-1$.

We sum the solutions of these four equations. For any even integer $n\geq 6$, $k=2$, the total number of pairs contributing to the $x^{2^n-2}$ is 
\[
1+1+2+1=5.
\]

Since the characteristic of the quotient ring is $2$, the coefficient of $x^{2^{n}-2}$ is $1$. Hence, in this case, $f(x)$ is not a permutation polynomial over $\mathbb{F}_{2^n}$.
\end{proof}

\section{Conclusions}
In this paper, we determined the permutation behavior of the polynomial function
$$
f(x)
=
\left(\frac{1}{x^2+x+\delta}\right)^{2^k}+x
$$
over $\mathbb{F}_{2^n}$, where
$\operatorname{Tr}_{\mathbb{F}_{2^n}/\mathbb{F}_2}(\delta)=1$. We proved that, for $0\leq k<n$, the function
$f(x)$ is a permutation of $\mathbb{F}_{2^n}$ if and only if
$k=0$ or $k=1$. This result confirms the conjecture suggested by a remark of Helleseth and Zinoviev \cite{HZ2003}, who observed from numerical evidence that the permutation property should occur only in the cases $k=0$ and $k=1$.




\bibliographystyle{elsarticle-num}\biboptions{sort&compress,longnamesfirst}
\bibliography{ffa-refs}
\end{document}